\documentclass{amsart}

\usepackage{color,amssymb}

\newcommand{\w}{\omega}
\newcommand{\IN}{\mathbb N}
\newcommand{\IZ}{\mathbb Z}
\newcommand{\IQ}{\mathbb Q}

\newcommand{\F}{\mathcal F}
\newcommand{\Ra}{\Rightarrow}
\newcommand{\dom}{\mathrm{dom}}

\newtheorem{theorem}{Theorem}[section]
\newtheorem{proposition}[theorem]{Proposition}
\newtheorem{corollary}[theorem]{Corollary}

\newtheorem{lemma}[theorem]{Lemma}
\newtheorem{problem}[theorem]{Problem}
\newtheorem{example}[theorem]{Example}

\theoremstyle{definition}
\newtheorem{definition}{Definition}
\newtheorem{remark}{Remark}

\title{On continuous self-maps and homeomorphisms of the Golomb space}
\author{Taras Banakh, Jerzy Mioduszewski and S\l awomir Turek}
\address{T.Banakh: Ivan Franko National University of Lviv (Ukraine) and Jan Kochanowski University in Kielce (Poland)}
\email{t.o.banakh@gmail.com}
\address{J.Mioduszewski: University of Silesia in Katowice (Poland)}
\email{miodusze@math.us.edu.pl}
\address{S.Turek: Cardinal Stefan Wyszy\'nski University in Warsaw (Poland)}
\email{s.turek@uksw.edu.pl}
\dedicatory{To the memory of Bohuslav Balcar (1943--2017)}

\begin{document}
\begin{abstract} The {\em Golomb space} $\IN_\tau$ is the set $\IN$ of positive integers endowed with the topology $\tau$ generated by the base consisting of arithmetic progressions $\{a+bn:n\ge 0\}$ with coprime $a,b$. We prove that the Golomb space $\IN_\tau$ has continuum many continuous self-maps, contains a countable disjoint family of infinite closed connected subsets, the set $\Pi$ of prime numbers is a dense metrizable subspace of $\IN_\tau$, and each homeomorphism $h$ of $\IN_\tau$ has the following properties: $h(1)=1$, $h(\Pi)=\Pi$ and $\Pi_{h(x)}=h(\Pi_x)$ for all $x\in\IN$. Here by $\Pi_x$ we denote the set of prime divisors of $x$.
\end{abstract}

\maketitle

In the AMS Meeting announcement \cite{Brown} M.~Brown introduced an amusing topology $\tau$ on the set $\IN$ of positive integers turning it into a connected Hausdorff space. The topology $\tau$ is generated by the base consisting of arithmetic progressions $a+b\IN_0:=\{a+bn:n\in\IN_0\}$ with coprime parameters $a,b\in\IN$. Here by $\IN_0=\{0\}\cup\IN$ we denote the set of non-negative integer numbers.

In \cite{SS} the topology $\tau$ is called the {\em relatively prime integer topology}. This topology was popularized by Solomon Golomb \cite{Golomb59}, \cite{Golomb61} who observed that the classical Dirichlet theorem on primes in arithmetic progressions is equivalent to the density of the set $\Pi$ of prime numbers in the topological space $(\IN,\tau)$. As a by-product of such popularization efforts, the topological space $\IN_\tau:=(\IN,\tau)$ is known in General Topology as the {\em Golomb space}, see \cite{Szcz}, \cite{Szcz13}.

The problem of studying the topological structure of the Golomb space was posed to the first author (Banakh) by the third author (Turek) in 2006. In its turn, Turek learned about this problem from the second  author (Mioduszewski) who listened to the lecture of Solomon Golomb on the first Toposym in 1961.

In this paper we study continuous self-maps and homeomorphisms of the Golomb space. In particular, we prove that the Golomb space has continuum many continuous self-maps and is not topologically homogeneous.







\section{Preliminaries and notations}

First let us fix notation. For a point $x\in\IN$ by $\tau_x=\{U\in\tau\colon x\in U\}$ we denote the family of open neighborhoods of $x$ in the topology $\tau$ of the Golomb space $\IN_\tau$.

For two numbers $x,y$ by $\gcd(x,y)$ we denote their greatest common divisor, and by $x\dag y$ the greatest divisor of $x$, which is coprime with $y$.

By $\Pi$ we denote the set of prime numbers. For a number $x\in\IN$ by $\Pi_x$ we denote the set of all prime divisors of $x$. Two numbers $x,y\in\IN$ are {\em coprime} iff $\Pi_x\cap\Pi_y=\emptyset$ (which is equivalent to saying that $\gcd(x,y)=1$).

A number $q\in\IN$ is called {\em square-free} if it is not divided by the square $p^2$ of any prime number $p$.

For a number $x\in\IN$ and a prime number $p$ let $l_p(x)$ be the largest integer number such that $p^{l_p(x)}$ divides $x$. The function $l_p(x)$ plays the role of logarithm with base $p$. A number $x$ is square-free if and only if $l_p(x)\le 1$ for any prime number $p$.

A function $f:X\to Y$ is called {\em finite-to-one} if for each $y\in Y$ the preimage $f^{-1}(y)$ is finite.

A family $\F$ of subsets of a set $X$ is called a {\em filter} if
\begin{itemize}
\item $\emptyset\notin\F$;
\item for any $A,B\in\F$ their intersection $A\cap B\in\F$;
\item for any sets $F\subset E\subset X$ the inclusion $F\in\F$ implies $E\in\F$.
\end{itemize}

In the subsequent proofs we shall exploit the  following two known results of Number Theory. The first one is a general version of the Chinese Remainder Theorem, which can be found in \cite[3.12]{J}.

\begin{theorem}[Chinese Remainder Theorem]\label{Chinese} For any numbers $a_1,\dots,a_n\in\IZ$ and $b_1,\dots,b_n\in\IN$ the following conditions are equivalent:
\begin{enumerate}
\item the intersection $\bigcap_{i=1}^n(a_i+b_i\IN)$ is not empty;
\item the intersection $\bigcap_{i=1}^n(a_i+b_i\IN)$ contains an infinite arithmetic progression;
\item for any $i,j$ the number $a_i-a_j$ is divisible by $\gcd(b_i,b_j)$.
\end{enumerate}
\end{theorem}

The second classical result is not elementary and is due to Dirichlet \cite[S.VI]{Dirichlet}, see also \cite[Ch.7]{Ap}.

\begin{theorem}[Dirichlet]\label{Dirichlet} Each arithmetic progression $a+b\IN$ with $\gcd(a,b)=1$ contains a prime number. \end{theorem}

\section{Superconnectedness of the Golomb space}

We define a topological space $X$ to be {\em superconnected} if for any non-empty open sets $U_1,\dots,U_n\subset X$ the intersection of their closures $\overline{U_1}\cap\dots\cap \overline{U_n}$ is not empty.

The proof of the following proposition is straightforward and is left to the reader as an exercise.

\begin{proposition}\label{supcont} \begin{enumerate}
\item Each superconnected space is connected.
\item The continuous image of a superconnected space is superconnected.
\end{enumerate}
\end{proposition}

 In this section we present some examples of superconnected subspaces of the Golomb space $\IN_\tau$.
But first describe the closures of  basic open sets in $\IN_\tau$. In fact, we will show a somewhat more general property.

\begin{lemma}\label{basic} For any $a,b\in\IN$ 
$$\overline{a+b\IN_0}=\IN\cap\bigcap_{p\in\Pi_b}\big(p\IN\cup (a+p^{l_p(b)}\IZ)\big).$$
\end{lemma}
\begin{proof}
First we prove that $\overline{a+b\mathbb N_0}\subset p\IN\cup(a+p^k\IZ)$ for every $p\in\Pi_b$ and $k=l_p(b)$.
Take any point $x\in\overline{a+b\IN_0}$ and assume that $x\notin p\IN$. Then $x+p^k\IN_0$ is a neighborhood of $x$ and hence the intersection $(x+p^k\IN_0)\cap(a+b\IN_0)$ is not empty. Then there exist $u,v\in\IN_0$ such that $x+p^ku=a+bv$. Consequently, $x-a=bv-p^ku\in p^k\IZ$ and $x\in a+p^k\IZ$ 

Next, take any point $x\in\IN\cap\bigcap_{p\in\Pi_b}\big(p\IN\cup(a+p^{l_p(b)}\IZ)\big)$. Given any basic neighborhood $x+d\IN_0$ of $x$, we should prove that $(x+d\IN_0)\cap(a+b\IN_0)\ne\emptyset$.

Our assumption guarantees that $x\in \bigcap_{p\in\Pi_b\setminus \Pi_x}(a+p^{l_p(b)}\IZ)=a+q\IZ$ where $q=\prod_{p\in \Pi_b\setminus\Pi_x}p^{l_p(b)}$. Since the numbers $x$ and $d$ are coprime, the greatest common divisor of $b$ and $d$ divides the number $q$. Since $x-a\in q\IZ$, the Euclides algorithm yields two numbers $u,v\in\IZ$ such that $x-a=bu-dv$, which implies that $(x+d\IZ)\cap(a+b\IZ)\ne\emptyset$ and so $(x+d\IN_0)\cap(a+b\IN_0)\ne\emptyset$.
\end{proof}

The next proposition  generalizes the connectedness of subspaces of the form $(1+p\IN_0)\cup p\IN$, where $p\in\Pi$ (proved in~\cite[Lemma 3.2]{Szcz13}).

\begin{proposition}\label{superconnect} For any sequences $\{a_i\}_{i\in\w}\subset\IN_0$ and $\{b_i\}_{i\in\w}\subset\IN$ with $a_0=0$ the subspace $$X=\bigcup_{i\in \w}(a_i+b_i\IN_0)$$ of $\IN_\tau$  is superconnected.
\end{proposition}

\begin{proof}   
	Given non-empty open sets $U_1,\dots,U_n\subset X$, we should prove that $X\cap \overline{U_1}\cap\dots\cap\overline{U_n}\ne\emptyset$.
	We can assume that each set $U_j$ is of the form
	$X\cap(c_j+d_j\IN)$ for some coprime numbers $c_j,d_j$. By the Chinese Remainder Theorem~\ref{Chinese}, $U_j$ contains an arithmetic progression $e_j+f_j\IN$ and by Lemma~\ref{basic}, $\overline{U_j}$ contains $\bigcap_{p\in\Pi_{f_j}}p\IN$. Since $a_{0}=0$ we have 
	$$\emptyset\ne b_0\IN_0\cap\bigcap_{j=1}^n\bigcap_{p\in \Pi_{f_j}}p\IN\subset X\cap \overline{U_1}\cap\dots\cap\overline{U_n}.$$
\end{proof}

\begin{corollary} The Golomb space is superconnected.
\end{corollary}

\section{Metrizability of the set of prime numbers in the Golomb space}

The main result of this section is the following theorem.

\begin{theorem} The set $\Pi$ of prime numbers is a dense metrizable subspace of the Golomb space $\IN_\tau$. Moreover, $\Pi$ is homeomorphic to the space $\IQ$ of rational numbers.
\end{theorem}

\begin{proof} The density of the set $\Pi$ in the Golomb space $\IN_\tau$ follows from Dirichlet's Theorem~\ref{Dirichlet}.

Next, we prove that the subspace $\Pi$ of $\IN_\tau$ is regular.
Given any number $x\in\Pi$ and an open neighborhood $O_x\subset\IN_\tau$ of $x$, we should find a neighborhood $U_x\subset\IN_\tau$ of $x$ such that $\Pi\cap \overline{U_x}\subset O_x$. We lose no generality assuming that $O_x=x+b\IN_0$ for some number $b$ such that $x\notin \Pi_b$ and $|\Pi_b|>1$. Choose $n\in\IN$ so large that $b^n>x$ and for any $p,r\in \Pi_b$ the difference $p-x$ is not divisible by $r^n$. We claim that the neighborhood $U_x:=x+b^n\IN_0$ has the required property: $\Pi\cap\overline{U_x}\subset O_x$.

Indeed, take any number $p\in \Pi\cap\overline{x+b^n\IN_0}$. By Lemma~\ref{basic}, $$\overline{x+b^n\IN_0}=\IN\cap\bigcap_{r\in \Pi_b}\big(r\IN\cup(x+r^{l_r(b^n)}\IZ)\big).$$  If $p\in\Pi_b$, then for any number $r\in\Pi_b\setminus\{p\}$, the choice of $n$ guarantees that $p\notin r\IN\cup(x+r^{l_r(b^n)}\IZ)$ and hence $p\notin \overline{x+b^n\IN_0}$, which is a contradiction. So, $p\notin\Pi_b$. In this case $p\in\bigcap_{r\in\Pi_b}(x+r^{l_r(b^n)}\IZ)=x+b^n\IZ$ and hence $p-x$ is divisible by $b^n$. Taking into account that $b^n>x$, we conclude that $p\ge x$ and hence $p\in x+b^n\IN_0\subset x+b\IN_0\subset O_x$. This completes the proof of the regularity of $\Pi$.

By the Tychonoff--Urysohn Metrization Theorem \cite[4.2.9]{En}, the second-countable regular space $\Pi$ is metrizable. The Dirichlet Theorem~\ref{Dirichlet} implies that the space $\Pi$ has no isolated points. By the Sierpi\'nski Theorem \cite[6.2.A(d)]{En}, $\Pi$ is homeomorphic to $\IQ$ (being a countable metrizable space without isolated points).
\end{proof}

In contrast to the set of prime numbers each basic open subset of the Golomb space is not regular  (but is totally disconnected). We recall that a topological space $X$ is {\em totally disconnected} if for any distinct points $x,y\in X$ there exists a closed-and-open set $U\subset X$ such that $x\in U$ and $y\notin U$.

\begin{proposition} For any coprime numbers $a,b\in\IN$ the subspace $X=a+b\IN_0$ of\/ $\IN_\tau$ is totally disconnected but not regular.
\end{proposition}

\begin{proof} First we show that the space $X=a+b\IN_0$ is not regular.
Choose any prime number $q\notin\Pi_b\cup\Pi_a$ and consider the
basic neighborhood $V=a+qb\IN_0\subset a+b\IN_0$ of the point $a$. Each basic neighborhood of $a$ which is  contained in $V$ has form $W=a+qbc\IN_0$ for some number $c\in\IN$, coprime with $a$. By  Lemma~\ref{basic} we have
\begin{multline*}
\overline{W}=\overline{a+qbc\IN_0}\supset\bigcap_{p\in \Pi_{qbc}}\big(p\IN\cup (a+p^{l_p(qbc)}\IN_0)\big)=\\=\bigcap_{p\in \Pi_{qbc}\setminus \Pi_b}\big(p\IN\cup (a+p^{l_p(qbc)}\IN_0)\big)\cap \bigcap_{p\in\Pi_b}\big(p\IN\cup (a+p^{l_p(qbc)}\IN_0)\big)\supset\\
\supset
\bigcap_{p\in \Pi_{qbc}\setminus \Pi_b}p\IN\cap \bigcap_{p\in\Pi_b}(a+p^{l_p(qbc)}\IN_0).\end{multline*}
The set 
$$\bigcap_{p\in \Pi_{qbc}\setminus \Pi_b}p\IN\cap \bigcap_{p\in\Pi_b}(a+p^{l_p(qbc)}\IN_0)$$
is non-empty by the Chinese Remainder Theorem~\ref{Chinese} and is contained in $q\IN\cap(a+b\IN_0)=q\IN\cap X$. However, $V\cap q\IN=(a+qb\IN_0)\cap q\IN=\emptyset$ 
because $a$ is not divisible by $q$. So, $\overline{W}\cap X\not\subset V$ and the space $X$ is not regular.

To see that the space $X=a+b\IN_0$ is totally disconnected, take any distinct points $x,y\in a+b\IN_0$ and choose $n\in\IN$ so large that $b^n$ does not divide $x-y$.
Observe that $\mathcal V=\{X\cap(z+b^n\IZ):z\in a+b\IN_0\}$ is a disjoint open cover of $X$, which implies that each set $V\in\mathcal V$ is open-and-closed in $X$. Moreover, since $b^n$ does not divide $x-y$, the points $x,y$ belong to distinct sets of the cover $\mathcal V$.  This implies that $X$ is totally disconnected.
\end{proof}

The last proposition gives another example of a space that is totally disconnected and not zero-dimensional (cf.~\cite[Ex. 1.2.15]{En1}).

\section{Continuous self-maps of the Golomb space}

In this section we study the structure of the set $C(\IN_\tau)$ of all continuous self-maps of the Golomb space $\IN_\tau$. In the following proposition the set $\IN^\IN$ of all self-maps of $\IN$ is endowed with the (Polish) topology of the Tychonoff product of discrete spaces $\IN$. 

Let us observe that a map $f\colon \IN_\tau\to\IN_\tau$ is continuous at a point $x\in\IN_\tau$ if and only if for every number $b$ coprime with $f(x)$ there is a number $a$ coprime with $x$ such that $f(x+a\IN_0)\subset f(x)+b\IN_0$.

\begin{proposition} The set $C(\IN_\tau)$ is an $F_{\sigma\delta}$-subset of the Polish space $\IN^\IN$.
\end{proposition}

\begin{proof} It is clear that $C(\IN_\tau)=\bigcap_{x\in\IN}C_x(\IN_\tau)$ where $C_x(\IN_\tau)$ denotes the set of all functions $f:\IN_\tau\to\IN_\tau$ which are continuous at $x$. In its turn, for every $x\in\IN$ the set $C_x(\IN_\tau)=\bigcap_{b\in\IN}C_{x,b}$ where $$C_{x,b}=\{f\in\IN^\IN:\gcd(f(x),b)\ne 1\}\cup C_{x,b}'$$ and $$
C_{x,b}':=\big\{f\in\IN^\IN:\exists a\in\IN\;\big(\gcd(a,x)=1\wedge (f(x+a\IN_0)\subset f(x)+b\IN_0)\big)\big\}.$$
Put $A_x:=\{a\in\IN:\gcd(a,x)=1\}$ and observe that $$
C_{x,b}'=\bigcup_{a\in A_x}\{f\in\IN^\IN:f(x+a\IN_0)\subset f(x)+b\IN_0\}
$$is a set of type $F_\sigma$ in $\IN^\IN$ and so is the set $C_{x,b}$. Then $C_x(\IN_\tau)=\bigcap_{b\in\IN}C_{x,b}$ is of type $F_{\sigma\delta}$ and so is the set $C(\IN_\tau)=\bigcap_{x\in\IN}C_x(\IN_\tau)$.
\end{proof}

Now we give a simple sufficient condition of continuity of a self-map of the Golomb space.

\begin{definition} A function $f:\dom(f)\to\IN$ defined of a subset $\dom(f)$ of $\IN$ is called {\em progressive} if
\begin{enumerate}
\item $\Pi_x\subset\Pi_{f(x)}$ for every $x\in\dom(f)$;
\item for any $x<y$ in $\dom(f)$ the number $f(y)-f(x)$ is divisible by $(y-x)\dag f(x)$.
\end{enumerate}
\end{definition}

We recall that for two numbers $x,y$ by $x\dag y$ we denote the greatest divisor of $x$ which is coprime with $y$.

\begin{proposition}\label{p:cont} Each finite-to-one progressive function $f:\IN_\tau\to\IN_\tau$ is continuous.
\end{proposition}

\begin{proof} Given a point $x\in\IN_\tau$ and a neighborhood $O_{f(x)}\in\tau$ of $f(x)$, we need to find a neighborhood $O_x\in\tau$ of $x$ such that $f(O_x)\subset O_{f(x)}$. We lose no generality assuming that $O_{f(x)}$ is of basic form $O_{f(x)}=f(x)+d\IN_0$. Since $\Pi_x\subset \Pi_{f(x)}$, the numbers $x$ and $d$ are coprime and hence $x+d\IN_0$ is a neighborhood of $x$ in $\IN_\tau$. Since the function $f$ is finite-to-one, the set $F= f^{-1}\big((f(x)+d\IZ)\setminus(f(x)+d\IN_0)\big)$ is finite and hence $O_x=(x+d\IN_0)\setminus F$ is a neighborhood of $x$ in the Golomb space $\IN_\tau$. It remains to prove that $f(O_x)\subset O_{f(x)}$. Given any $y\in O_x$, we need to show that $f(y)\in f(x)+d\IZ$. This is trivially true if $y=x$. So, we assume that $y\ne x$ and hence $y\in x+d\IN$. Then $d$ divides $y-x$. Since $d$ and $f(x)$ are coprime, $d$ divides the number $b=(y-x)\dag  f(x)$. Taking into account that the function $f$ is progressive, we conclude that $f(y)-f(x)\in b\IZ\subset d\IZ$.
\end{proof}

For polynomials with non-negative integer coefficients the equivalence of conditions $(1)$--$(3)$ in the following theorem was proved by Szczuka \cite[Theorem 4.3]{Szcz13}.

\begin{theorem}\label{poly} For a non-constant polynomial $f:\IN\to\IN$, $f:x\mapsto a_0+a_1x+\dots+a_nx^n$ with integer coefficients the following conditions are equivalent:
\begin{enumerate}
\item $a_0=0$;
\item $f$ is a continuous self-map of the Golomb space $\IN_\tau$;
\item for any connected subspace $C\subset \IN_\tau$ the image $f(C)$ is connected;
\item for any superconnected subspace $C\subset \IN_\tau$ the image $f(C)$ is superconnected;
\item for any superconnected subspace $C\subset \IN_\tau$ the image $f(C)$ is connected.
\end{enumerate}
\end{theorem}

\begin{proof} To prove the implication $(1)\Ra(2)$, assume that $a_0=0$. In this case the polynomial $f(x)$ is divisible by $x$ and $f(x)-f(y)$ is divisible by $x-y$. These observations imply that the function $f$ is progressive. Since $f$ is not constant, it is finite-to-one. Applying Proposition~\ref{p:cont}, we conclude that $f$ is continuous.
\smallskip

The implications $(3)\Leftarrow(2)\Ra(4)$ and $(3)\Ra(5)\Leftarrow(4)$ follow from Proposition~\ref{supcont}. So, it remains to prove that $(5)\Ra(1)$. To derive a contradiction, assume that $a_0\ne 0$. Since $f$ is not constant, there exists $x\in\IN$ such that $f(x)\ne a_0$. Choose any prime number $p>\max\{a_0,x,f(x)\}$. By Proposition~\ref{superconnect}, the subspace $X=p\IN\cup(x+p\IN_0)$ is superconnected. On the other hand, its image $f(X)$  can be written as the union $f(X)=U\cup V$ of two non-empty disjoint open subspaces $U=f(p\IN)=f(X)\cap (a_0+p\IZ)$ and $V=f(x+p\IN)=f(X)\cap(f(x)+p\IZ)$ of $f(X)$.
\end{proof}

The following example shows that Theorem~\ref{poly} cannot be extended to polynomials with rational coefficients.

\begin{example} The polynomial $f:\IN_\tau\to\IN_\tau$, $f:x\mapsto \frac12(x+x^2)$, is discontinuous.
\end{example}

\begin{proof} It suffices to check that $f$ is discontinuous at $x=2$. Assuming that $f$ is continuous at $2$, for the neighborhood $3+2\IN_0$ of $3=f(2)$, we can find a basic neighborhood $2+b\IN_0$ of $2$ such that $f(2+b\IN_0)\subset 3+2\IN_0$. The number $b$, being coprime with $2$, is odd and hence can be written as $b=2n-1$ for some $n\in\IN$. Consider the number $4n=2+2(2n-1)=2+2b\in 2+b\IN_0$ and observe that $f(4n)=2n(4n+1)\notin 3+2\IN_0$, which contradicts the choice of $b$.
\end{proof}

\begin{corollary} The Golomb space contains a countable family of pairwise disjoint closed infinite superconnected subspaces.
\end{corollary}

\begin{proof} By Theorem~\ref{poly}, for every $n\in\IN$ the polynomial $f_n:\IN_\tau\to\IN_\tau$, $f_n:x\mapsto x^2+nx$, is a continuous self-map of the Golomb space $\IN_\tau$.

By Lemma~\ref{l:frob} proved below, the set $f_n(\IN)$ is closed in the Golomb space $\IN_\tau$. Choose any prime number $p_n> n^2+n$ and observe that the set $p_n\IN$ is closed in the Golomb space $\IN_\tau$. Then the intersection $X_n=p_n\IN\cap f_n(\IN)$ is closed in $\IN_\tau$, too. By Proposition~\ref{superconnect}, the set $$f_n^{-1}(X_n)=f_n^{-1}(p_n\IN)=\{x\in \IN:x(x+n)\in p_n\IN\}=p_n\IN\cup(p_n-n+p_n\IN_0)$$ is superconnected and so is its image $f_n(f_n^{-1}(X_n))=X_n$.

It remains to prove that for any numbers $n<m$ the sets $X_n$ and $X_m$ are disjoint. Assuming that $X_n\cap X_m$ contains some number $z$, we conclude that $z\in X_n\cap X_m\subset p_n\IN\cap p_m\IN\subset p_m\IN$.
Find numbers $x,y\in\IN$ such that $x^2+nx=f_n(x)=z=f_m(y)=y^2+my$ and observe that $$(2x+n)^2-n^2=4x^2+4nx=4y^2+4my=(2y+m)^2-m^2.$$
Then $(2(y-x)+m-n)(2(y+x)+m+n)=(2y+m)^2-(2x+n)^2=m^2-n^2$ and hence $$y\le 2(y+x)+m+n\le m^2-n^2<m^2.$$ On the other hand, $y(y+m)=z\in p_m\IN$ implies that $y$ or $y+m$ is divisible by $p_m$ and hence $y\ge p_m-m>m^2$ by the choice of prime number $p_m$.  This contradiction shows that $X_n\cap X_m=\emptyset$.
\end{proof}

\begin{lemma}\label{l:frob} For any $n\in\IN_0$ the set $X:=\{x^2+nx:x\in \IN\}$ is closed in the Golomb space $\IN_\tau$.
\end{lemma}

\begin{proof} Given any $a\in \IN\setminus X$, it suffices to find a prime number $p>a$ such that $a+p\IN_0$ is disjoint with $X$. Consider the polynomial $f(x):=x^2+nx-a$ and observe that its positive root $\frac{-n+ \sqrt{n^2+4a}}{2}$ is not integer (as $a\notin X$). This implies that $f$ has no rational roots and hence $f$ is irreducible over the field $\IQ$. By the classical Frobenius Density Theorem (see \cite{Sury} or \cite{SL}), there exist infinitely many prime numbers $p$ such that the polynomial $f$ has no roots in $\IZ/p\IZ$. For any such prime number $p$ we have $(a+p\IN)\cap X=\emptyset$.
\end{proof}

Lemma~\ref{l:frob} implies that the set $\{x^2:x\in\IN\}$ is closed in $\IN_\tau$. On the other hand, we have the following fact.

\begin{proposition} For any $n\in\IN$ the set $X_{8n}:=\{x^{8n}:x\in\IN\}$ is not closed in the Golomb space $\IN_\tau$.
\end{proposition}

\begin{proof} First we show that the set $X_8=\{x^8:x\in\IN\}$ is not closed in $\IN_\tau$. For this purpose we shall exploit a well-known Wang Counterexample \cite{Wang} saying that the equation $x^8=16$ has no integer solutions but has solutions in any field of odd prime order.

We shall show that $16\in \overline{X_8}\setminus X_8$. Given any
neighborhood $O_{16}\subset \IN_\tau$ of $16$, we should prove that
$O_{16}\cap X_8\ne\emptyset$. By the definition of the topology $\tau$, there exists an odd number $b\in\IN$ such that $16+b\IN_0\subset O_{16}$. Observe that $$x^{8}-16=(x^2-2)(x^2+2)(x^2-2x+2)(x^2+2x+2).$$

 Theorems 9.3, 9.4 and 9.5 in \cite{Ap} imply that for every odd prime number $p$ one of the numbers $2,-2,-1$ is a square in field $\IZ/p\IZ$.
If some $x\in\IZ/p\IZ$ has $x^2=\pm2$, then $x^8=16$. If $x^2=-1$, then $(1+x)^2-2(1+x)+2=0$ and hence $(1+x)^8=16$. In any case, for any odd prime number $p$ the polynomial $f(x):=x^8-16$ has a root in the field $\IZ/p\IZ$. By induction we shall show that this polynomial has a root in the residue rings $\IZ/p^k\IZ$ for all $k\in\IN$. Assume that for some $k\in\IN$ we have found a number $s\in\IZ$ such that $f(s)\in p^k\IZ$. We claim that $f'(s)=8x^7\notin p\IZ$. Otherwise $s$ would be divisible by $p$ and then $s^8-16$ cannot be divisible by $p^k$, which is a contradiction. So, $f'(s)\notin p\IZ$ and we can apply Theorem 5.30 in \cite{Ap} to find a number $r\in\IZ$ such that $f(r)\in p^{k+1}\IZ$, which implies that the equation $x^8-16=0$ has a solution in the residue ring $p^{k+1}\IZ$.

So, for every prime divisor $p$ of $b$ we can find a number $x_p\in\IN$ such that $x_p^8-16\in p^{l_p(b)}\IN_0$.
Using the Chinese Remainder Theorem~\ref{Chinese}, find a number $x\in\IN$ such that $x\ge 16$ and $x\in x_p+p^{l_p(b)}\IZ$ for every $p\in\Pi_b$.
Then $$x^8-16\in \bigcap_{p\in\Pi_b}x^8-16+p^{l_p(b)}\IZ=
\bigcap_{p\in\Pi_b}x_p^8-16+p^{l_p(b)}\IZ=
\bigcap_{p\in\Pi_b}p^{l_p(b)}\IZ=b\IZ$$and hence $x^8\in X_8\cap (16+b\IN_0)$. So, $(16+b\IN_0)\cap X_8\ne\emptyset$.

Now we prove that for any $n\in\IN$ the set $X_{8n}=\{x^{8n}:x\in\IN\}$ is not closed in $\IN_\tau$. By Theorem~\ref{poly}, the polynomial map $f\colon \IN_\tau\to\IN_\tau$, $f\colon x\mapsto x^n$, is continuous. Taking into account that this map is injective and $f(X_8)=X_{8n}$, we conclude that $$16^n=f(16)\in f(\overline{X_8}\setminus X_8)\subset \overline{f(X_8)}\setminus f(X_8)=\overline{X_{8n}}\setminus X_{8n}.$$
\end{proof}

Now we show that the set $C(\IN_\tau)$ of all continuous self-maps of the Golomb space has cardinality of continuum.

\begin{theorem} The set $C(\IN_\tau)$ contains a subset $\partial T$ of cardinality continuum which is closed in $\IN^\IN$.
\end{theorem}

\begin{proof} For every $n\in\IN$ let $T_n$ be the family of increasing progressive functions $f:[1,n]\to\IN$ defined on the interval $[1,n]:=\{1,\dots,n\}$. Let $T_0$ be the singleton consisting of the unique function $f_0:\emptyset\to\IN$. On the union $T=\bigcup_{n\in\w}T_n$ consider the partial order $\le$ defined by $f\le g$ iff $\dom(f)\subset \dom(g)$ and $f=g|\dom(f)$. It is clear that this partial order turns $T$ into a tree. The set $\partial T=\{f\in\IN^\IN:\forall n\in\IN\;f|[1,n]\in T_n\}$ is closed in $\IN^\IN$ and can be identified with the set of branches of the tree $T$. Since each function $f\in\partial T$ is increasing and progressive, Proposition~\ref{p:cont} guarantees that $f$ is continuous and hence $\partial T\subset C(\IN_\tau)$.

It remains to check that $|\partial T|=\mathfrak c$. This equality will follow as soon as for any $n\in\IN$ and $f\in T_n$ we show that the set $\mathrm{succ}(f)=\{g\in T_{n+1}:g|[1,n]=f\}$ of successors of $f$ in the tree $T$ contains more than one point. We shall prove more: the set $\mathrm{succ}(f)$ is infinite.

Observe that a function $g:[1,n+1]\to\IN$ belongs to $\mathrm{succ}(f)$ if and only if $g|[1,n]=f$, $g(n+1)>g(n)$ and $g(n+1)$ belongs to the set
$$Y_f:=\bigcap_{p\in\Pi_{n+1}}p\IN\cap 
\bigcap_{k\in[1,n]}\big(f(k)+\big((n+1-k)\dag  f(k)\big)\IZ\big).$$ So, it suffices to prove that the set $Y_f$ is infinite. By the Chinese Remainder Theorem~\ref{Chinese}, the set $Y_f$ is infinite if and only if
\begin{enumerate}
\item for every $p\in\Pi_{n+1}$ and $k\in[1,n]$ the number $f(k)$ is divisible by $\gcd\big(p,(n+1-k)\dag  f(k)\big)$;
\item for any numbers $k<l$ in $[1,n]$ the number $f(l)-f(k)$ is divisible by $\gcd\big((n+1-k)\dag  f(k),(n+1-l)\dag  f(l)\big)$.
\end{enumerate}

To verify the first condition, fix any prime number $p\in\Pi_{n+1}$ and any $k\in[1,n]$. We claim that $\gcd(p,(n+1-k)\dag  f(k))=1$. If $p\notin\Pi_k$, then $p$ does not divide $n+1-k$ and hence $\gcd(p,(n+1-k)\dag  f(k))=1$. If $p\in\Pi_k$, then $p\in\Pi_{f(k)}$ and hence $\gcd(p,(n+1-k)\dag  f(k))=1$. In both cases $\gcd(p,(n+1-k)\dag  f(k))=1$ divides $f(k)$.

To verify the second condition, fix any numbers $k<l$ in $[1,n]$.
Let $d$ be the largest common divisor of $(n+1-k)\dag  f(k)$ and $(n+1-l)\dag  f(l)$. Then $d$ is coprime with $f(k)$ and divides both the numbers $n+1-k$ and $n+1-l$, so $d$ divides their difference $(n+1-k)-(n+1-l)=l-k$. Consequently, $d$ divides $(l-k)\dag  f(k)$. Since $f$ is progressive, $f(l)-f(k)$ is divisible by $(l-k)\dag  f(k)$ and hence is divisible by $d$.
\end{proof}

\begin{problem} Is the set $C(\IN_\tau)$ dense in $\IN^\IN$ \textup{(}or in $\IN_\tau^\IN$\textup{)}?
\end{problem}

\section{Homeomorphisms of the Golomb space}

In this section we study homeomorphisms of the Golomb space $\IN_\tau$ and prove the following main theorem.

\begin{theorem}\label{mainH} Each homeomorphism $h:\IN_\tau\to\IN_\tau$ of the Golomb space has the following properties:
\begin{enumerate}
\item $h(1)=1$;
\item $h(\Pi)=\Pi$;
\item $\Pi_{h(x)}=h(\Pi_x)$ for every $x\in\IN$;
\item there exists a multiplicative bijection $\mu$ of $\IN$ such that $h(x^n)=h(x)^{\mu(n)}$ for all $x,n\in\IN$.
\end{enumerate}
\end{theorem}

We recall that for a number $x\in\IN$ by $\Pi_x$ the set of all prime divisors of $x$ is denoted. Also by $\tau$ we denote the topology of the Golomb space $\IN_\tau$ and by $\tau_x$ the family of open neighborhoods of a point $x\in\IN$ in the Golomb space $\IN_\tau$. A function $\mu:\IN\to\IN$ is {\em multiplicative} if $\mu(xy)=\mu(x){\cdot}\mu(y)$ for any $x,y\in\IN$. 

The four items of Theorem~\ref{mainH} are proved in Lemmas~\ref{fix1}, \ref{Pi}, \ref{final3}, and \ref{l:finall}, respectively.

The superconnectedness of the Golomb space implies that the family $$\F_0=\Big\{F\subset\IN\colon\exists U_1,\dots,U_n\in\tau\setminus\{\emptyset\}\mbox{ with }\bigcap_{i=1}^n\overline{U_i}\subset F\Big\}$$
is a filter on $\IN$.

The definition of $\F_0$ implies that this filter is preserved by any homeomorphism $h$ of $\IN_\tau$ (which means that the filter $h[\F_0]:=\{h(F):F\in\F_0\}$ coincides with $\F_0$).

\begin{lemma}\label{filter0} The filter $\F_0$ is generated by the base consisting of the sets $q\IN$ for a square-free number $q\in\IN$.
\end{lemma}

\begin{proof} Lemma~\ref{basic} implies that each element $F\in\F_0$ contains the set $q\IN$ for some square-free number $q$.
It remains to show that for each square-free number $q>1$ the set $q\IN$ is contained in the filter $\F_0$. This is proved in the following lemma.
\end{proof}

\begin{lemma}\label{special1} For two distinct numbers $x,y\in \IN$ and a square-free number $q>1$ the following conditions are equivalent:
\begin{enumerate}
\item $q$ is coprime with $x$ and $y$;
\item there are open sets $U_x\in\tau_x$ and $U_y\in\tau_y$ such that $\overline{U_x}\cap\overline{U_y}= q\IN$;
\item there are open sets $U_x\in\tau_x$ and $U_y\in\tau_y$ such that $\overline{U_x}\cap\overline{U_y}\subset q\IN$.
\end{enumerate}
\end{lemma}

\begin{proof} To prove the implication $(1)\Rightarrow (2)$, assume that a square-free number $q>1$ is coprime with $x$ and $y$. Choose $n\in\IN$ so large that for any prime number $p\in\Pi_q$ the difference $x-y$ is not divided by $p^n$. Then for the neighborhoods $U_x=x+q^n\IN_0$ and $U_y=y+q^n\IN_0$ we get
$$\overline{U}_x\cap\overline{U}_y=\overline{x+q^n\IN_0}\cap\overline{y+q^n\IN_0}=
\IN\cap\bigcap_{p\in\Pi_q}\big(p\IN\cup(x+p^n\IZ)\big)\cap\big(p\IN\cup (y+p^n\IZ)\big)=\bigcap_{p\in\Pi_q}p\IN=q\IN.$$
The implication $(2)\Ra(3)$ is trivial. To prove the implication $(3)\Ra(1)$, choose two open sets $U_x\in\tau_x$ and $U_y\in\tau_y$ with $\overline{U_x}\cap\overline{U_y}\subset q\IN$. 
We can assume that the open sets $U_x$ and $U_y$ are of basic form $U_x=x+b\IN_0$ and $U_y=y+d\IN$. 
To derive a contradiction, assume that $q$ has a common prime divisor $p$ with $x$ or $y$. Without loss of generality $p\in \Pi_x$ and hence $p\notin \Pi_{b}$. If also $p\notin\Pi_{d}$, then by the Chinese Remainder Theorem and Lemma~\ref{basic}, 
$$\emptyset\ne (1+p\IN)\cap\bigcap_{r\in\Pi_{b}\cup\Pi_d}r\IN\subset\overline{U_x}\cap\overline{U_y}\setminus q\IN.$$ If $p\in\Pi_d$, then $p\notin\Pi_y$   and by the Chinese Remainder Theorem and Lemma~\ref{basic}, 
$$\emptyset\ne (y+p^{l_p(d)}\IN)\cap\bigcap_{r\in\Pi_{b}\cup\Pi_d\setminus\{p\}}r\IN\subset\overline{U}_x\cap\overline{U}_y\setminus q\IN.$$In both cases we obtain a contradiction with our assumption $\overline{U}_x\cap\overline{U}_y\subset q\IN$.
\end{proof}

\begin{lemma}\label{fix1} The number $1$ is a fixed point of any homeomorphism $h$ of the Golomb space $\IN_\tau$.
\end{lemma}

\begin{proof} To derive a contradiction, assume that $x=h(1)$ is not equal to $1$.  Fix any $p\in\Pi_x$. By Lemma~\ref{filter0}, the set $p\IN$ belongs to the filter $\F_0$. Since the filter $\F_0$ is invariant under homeomorphisms of $\IN_\tau$, $h^{-1}(p\IN)\in\F_0$. By Lemma~\ref{filter0}, there exists a square-free number $q$ such that $q\IN\subset h^{-1}(p\IN)$. Choose any $z\in\IN\setminus\{1\}$ coprime with $q$. By Lemma~\ref{special1}, the Golomb space $\IN_\tau$ contains open sets  $U_z\ni z$ and $U_1\ni 1$ such that $\overline{U_z}\cap\overline{ U_1}\subset q\IN$. Then for the point $y=h(z)$, the sets $V_y=h(U_z)\ni y$ and $V_x=h(U_1)\ni x$ have the property $\overline{V_y}\cap\overline{V_x}\subset h(q\IN)\subset p\IN$. Now Lemma~\ref{special1} implies that $p\notin\Pi_x$, which contradicts the choice of $p$.
\end{proof}

For any point $x\in\IN$ consider the filter $$\F_x=\{F\subset\IN:\exists U_x\in\tau_x,\;\exists U_1\in\tau_1\mbox{ such that }\overline{U}_x\cap\overline{ U}_1\subset F\}.$$
The definition of the filter $\F_x$ and Lemma~\ref{special1}
imply

\begin{lemma}\label{l:inter} For any point $x\in\IN$ we get 
\begin{enumerate}
\item $\Pi_x=\{p\in\Pi:p\IN\notin\F_x\}$;
\item $\F_x=\{F\subset\IN:\exists a,b\in\IN\mbox{ such that $\Pi_b\cap\Pi_x=\emptyset$ and $\overline{1+a\IN}\cap\overline{x+b\IN}\subset F$}\}$.
\end{enumerate}
\end{lemma}

The interplay between $\Pi_x$ and $\F_x$ from Lemma~\ref{l:inter} implies

\begin{lemma}\label{equiv} For two numbers $x,y\in\IN\setminus\{1\}$ the following conditions are equivalent:
\begin{enumerate}
\item $\F_x\subset\F_y$;
\item $\Pi_y\subset\Pi_x$.
\end{enumerate}
\end{lemma}

\begin{lemma}\label{order} For every homeomorphism $h$ of the Golomb space $\IN_\tau$, and any numbers $x,y\in\IN\setminus\{1\}$ with $\Pi_x\subset\Pi_y$ we get  $\Pi_{h(x)}\subset\Pi_{h(y)}$.
\end{lemma}

\begin{proof} Assume that $\Pi_x\subset\Pi_y$. By Lemma~\ref{equiv},  $\F_y\subset\F_x$. Lemma~\ref{fix1} and the (topological) definition of the filters $\F_x$ and $\F_{h(x)}$ imply that $h[\F_x]=\F_{h(x)}$ where $h[\F_x]=\{h(F):F\in\F_x\}$. By analogy we can show that $h[\F_y]=\F_{h(y)}$. Then $\F_{h(y)}=h[\F_y]\subset h[\F_x]=\F_{h(x)}$ and by Lemma~\ref{equiv}, $\Pi_{h(x)}\subset\Pi_{h(y)}$.
\end{proof}

\begin{lemma}\label{sigma} For every homeomorphism $h:\IN_\tau\to\IN_\tau$ there exists a unique bijective map $\sigma:\Pi\to\Pi$ of the set $\Pi$ of all prime numbers such that $\Pi_{h(p)}=\{\sigma(p)\}$ and $\Pi_{h^{-1}(q)}=\{\sigma^{-1}(q)\}$ for any $p,q\in \Pi$.
\end{lemma}

\begin{proof} By Lemma~\ref{fix1}, for every prime number $p$  the image $h(p)$ is not equal to $1$, which implies that the set $\Pi_{h(p)}$ is not empty and hence contains some prime number $\sigma(p)$.
We claim that $\Pi_{h(p)}=\{\sigma(p)\}$. Since $\Pi_{\sigma(p)}=\{\sigma(p)\}\subset \Pi_{h(p)}$, we can apply Lemma~\ref{order} and conclude that $\Pi_{h^{-1}(\sigma(p))}\subset\Pi_p=\{p\}$ and hence $\Pi_{h^{-1}(\sigma(p))}=\{p\}$.
Since $\Pi_{p}=\{p\}\subset\Pi_{h^{-1}(\sigma(p))}$, we can apply Lemma~\ref{order} once more and conclude that $\Pi_{h(p)}\subset\Pi_{\sigma(p)}=\{\sigma(p)\}$ and hence $\Pi_{h(p)}=\{\sigma(p)\}$.

Next, we show that the map $\sigma$ is bijective. The injectivity of $\sigma$ follows from the equality $\Pi_{h^{-1}(\sigma(p))}=\{p\}$ holding for every $p\in\Pi$. To see that $\sigma$ is surjective, take any prime number $q$ and choose any prime number $p\in\Pi_{h^{-1}(q)}$.
Since $\Pi_p\subset\Pi_{h^{-1}(q)}$, we can apply Lemma~\ref{order} and conclude that $\{\sigma(p)\}=\Pi_{h(p)}\subset \Pi_q=\{q\}$ and hence $q=\sigma(p)$. Then $$\Pi_{h^{-1}(q)}=\Pi_{h^{-1}(\sigma(p))}=\{p\}=\{\sigma^{-1}(q)\}.$$

The equality $\Pi_{h(p)}=\{\sigma(p)\}$ holding for every $p\in\Pi$ witnesses that the map $\sigma:\Pi\to\Pi$ is uniquely determined by the homeomorphism $h$.
\end{proof}

Lemma~\ref{sigma} admits a self-improvement:

\begin{lemma}\label{sigma2} For every homeomorphism $h:\IN_\tau\to\IN_\tau$ there exists a unique bijective map $\sigma:\Pi\to\Pi$ such that $\Pi_{h(x)}=\sigma(\Pi_x)$ and $\Pi_{h^{-1}(x)}=\sigma^{-1}(\Pi_x)$ for any $x\in \IN$.
\end{lemma}

\begin{proof} By Lemma~\ref{sigma}, for every homeomorphism $h:\IN_\tau\to\IN_\tau$ there exists a unique bijective map $\sigma:\Pi\to\Pi$ of such that $\Pi_{h(p)}=\sigma(p)$  and $\Pi_{h^{-1}(p)}=\sigma^{-1}(p)$ for any $p\in \Pi$.

We claim that $\Pi_{h(x)}=\sigma(\Pi_x)$ for any number $x\in\IN$. If $x=1$, then this follows from Lemma~\ref{fix1}. So, we assume that $x\ne1$. For every prime number $p\in\Pi_x$ the inclusion $\Pi_p=\{p\}\subset \Pi_x$ and Lemma~\ref{order} imply
$\{\sigma(p)\}=\Pi_{h(p)}\subset\Pi_{h(x)}$. So, $\sigma(\Pi_x)\subset\Pi_{h(x)}$.

On the other hand, for any prime number $q\in\Pi_{h(x)}$, we can apply Lemma~\ref{sigma} to the homeomorphism $h^{-1}$ and conclude that the set $\Pi_{h^{-1}(q)}$ coincides with the singleton $\{p\}$ of some prime number $p$. Taking into account that $\Pi_p=\{p\}=\Pi_{h^{-1}(q)}$ and applying Lemma~\ref{order}, we conclude that $\{\sigma(p)\}=\Pi_{h(p)}=\Pi_{q}=\{q\}\subset \Pi_{h(x)}$. Applying Lemma~\ref{order} to the inclusion $\Pi_{q}\subset\Pi_{h(x)}$, we get the inclusion $\{p\}=\Pi_{h^{-1}(q)}\subset \Pi_{h^{-1}(h(x))}=\Pi_x$ and finally $q=\sigma(p)\in\sigma(\Pi_x)$. So, $\Pi_{h(x)}\subset \sigma(\Pi_x)$ and hence $\Pi_{h(x)}=\sigma(\Pi_x)$. Also
$\Pi_{h^{-1}(q)}=\{p\}=\{\sigma^{-1}(q)\}$.
\end{proof}

For a number $x\in\IN$ we shall denote by $x^{\IN}:=\{x^n:n\in\w\}$ the multiplicative semigroup in $\IN$ generated by $x$. In this case we say that $x^{\IN}$ is the {\em monogenic semigroup generated by} $x$. For $x,k\in\IN$ the monogenic semigroup $(x^k)^\IN$ generated by $x^k$ will be denoted by $x^{k\IN}$. 

\begin{lemma}\label{l:monogen} Each homeomorphism $h$ of the Golomb space $\IN_\tau$ preserves monogenic semigroups in the sense that $$h(a^{\IN})=h(a)^{\IN}$$ holds for all $a\in\IN$.
\end{lemma}

\begin{proof} Fix any point $a\in\IN$ and put $b:=h(a)$. First we show that $h(a^{\IN})\subset b^{\IN}$. To derive a contradiction, assume that $h(a^n)\notin b^{\IN}$ for some $n\ge 2$. 

By Lemma~\ref{sigma2}, there exists a bijective function $\sigma:\Pi\to\Pi$ such that $\Pi_{h(x)}=\sigma(\Pi_x)$ for all $x\in \IN$. 
By Theorem~\ref{poly}, the polynomial $f:\mathbb {N}_\tau\to\mathbb{N}_\tau$, $f:x\mapsto x^n$, is continuous. Then the map $\varphi=h\circ f\circ h^{-1}:\mathbb{N}_\tau \to\mathbb {N}_\tau$ is continuous, too. Observe that $\varphi(b)=h\circ f(a)=h(a^n)$ and $$\Pi_{\varphi(b)}=\Pi_{h(a^n)}=\sigma(\Pi_{a^n})=\sigma(\Pi_a)=\Pi_{h(a)}=\Pi_b.$$
Then for every $k\in\IN$ the numbers $h(a^n)$ and $b^k-1$ are coprime. 
Choose a number $k\in\IN$ such that $b^k-1>h(a^n)$ and consider the neighborhood $h(a^n)+(b-1)\IN_0$ of $h(a^n)=\varphi(b)$. By the continuity of the map $\varphi$, the point $b$ has a neighborhood $b+d\IN_0\in\tau_b$ such that $\varphi(b+d\IN_0)\subset h(a^n)+(b^k-1)\IN_0$.

By the Dirichlet Theorem~\ref{Dirichlet}, the arithmetic progression $b+(b^k-1)d\IN_0$ contains a prime number $p$. Lemma~\ref{sigma} implies that $\varphi(p)=p^l$ for some $l\ge 1$.  Then  $$h(a^n)=\varphi(p)=p^l\in(b+(b^k-1)\IN_0)^l\subset b^l+(b^k-1)\IZ.$$ Write $l$ as $l=ki+j$ where $j\in[0,k)$. If $i=0$, then $l=j$ and hence $h(a^n)\in b^j+(b^k-1)\IZ$. If $i>0$, then $b^l-b^j=b^j((b^k)^i-1)\in(b^k-1)\IZ$ and again $h(a^n)\in b^l+(b^k-1)\IZ=b^j+(b^k-1)\IZ$. In both cases we obtain the inclusion $h(a^n)\in b^j+(b^k-1)\IZ$, which is not possible as $0<|b^j-h(a^n)|\le \max\{h(a^n),b^{k-1}\}<b^k-1$. This contradiction completes the proof of the inclusion $h(a^n)\subset h(a)^{\IN}$.

By analogy we can prove that $h^{-1}(x^{\IN})\subset (h^{-1}(x))^{\IN}$ and hence $x^{\IN}\subset h(h^{-1}(x)^{\IN})$ for all $x\in\IN$. Then for any $a,n\in\IN$ we get $h(a)^n\in h(h^{-1}(h(a))^{\IN}=h(a^{\IN})$.
So, $h(a)^{\IN}\subset h(a^{\IN})$. Combining this inclusion with $h(a^{\IN})\subset h(a)^{\IN}$, we obtain the required equality $h(a^{\IN})=h(a)^{\IN}$.
\end{proof}

\begin{lemma}\label{Pi} $h(\Pi)=\Pi$ for any homeomorphism $h$ of the Golomb space $\IN_\tau$.
\end{lemma}

\begin{proof} It suffices to show that $h(p)\in\Pi$ for every $p\in\Pi$. By Lemma~\ref{sigma2}, there exists a bijective map $\sigma:\Pi\to\Pi$ such that $\sigma(\Pi_{h(x)})=\sigma(\Pi_x)$ and $\sigma^{-1}(\Pi_{h^{-1}(x)})=h^{-1}(\sigma^{-1}(\Pi_x))$ for all $x\in \IN$. In particular, $\Pi_{h(p)}=\{q\}$ for $q=\sigma(p)$ and $\Pi_{h^{-1}(q)}=\{p\}$. This implies that
$h(p)=q^n$ and $h^{-1}(q)=p^m$ for some $n,,m\in\IN$.
Applying Lemma~\ref{l:monogen}, we obtain 
$$q^{\IN}=h(p^m)^\IN=h(p^{m\IN})\subset h(p^\IN)=h(p)^\IN=q^{n\IN}$$and hence $n=1$.
\end{proof}

\begin{lemma}\label{final3} For any homeomorphism $h$ of the Golomb space $\IN_\tau$ we have $\Pi_{h(x)}=h(\Pi_x)$ for every $x\in\IN$.
\end{lemma}

\begin{proof} By Lemma~\ref{sigma2}, there exists a bijective function $\sigma:\Pi\to\Pi$ such that $\Pi_{h(x)}=\sigma(\Pi_x)$ for all $x\in \IN$. By Lemma~\ref{Pi}, for every prime number $p$ we get $\{h(p)\}=\Pi_{h(p)}=\{\sigma(p)\}$. Therefore, $\sigma=h|\Pi$ and
$\Pi_{h(x)}=\sigma(\Pi_x)=h(\Pi_x)$ for every $x\in \IN$.
\end{proof}

Next, we investigate the restrictions of homeomorphisms of $\IN_\tau$ to monogenic subsemigroups of $\IN$.

A function $\mu:\IN\to\IN$ is called {\em multiplicative} if $\mu(xy)=\mu(x)\cdot\mu(y)$ for all $x,y\in\IN$. It is easy to see that a function is multiplicative if and only if 
\begin{itemize}
\item $\mu(p^n)=\mu(p)^n$ for every prime number $p\in\Pi$ and every $n\in\IN$;
\item $\mu(x{\cdot}y)=\mu(x)\cdot\mu(y)$ for any coprime numbers $x,y\in \IN$. 
\end{itemize}
This implies that each multiplicative function $\mu$ is uniquely determined by its restriction $\mu{\restriction}\Pi$ to the set $\Pi$ of prime numbers. If a multiplicative function $\mu:\IN\to\IN$ is bijective, then $\mu(\Pi)=\Pi$ and the inverse function $\mu^{-1}$ is multiplicative, too.

\begin{lemma}\label{l:multi} Let $h$ be a homeomorphism of the Golomb space $\IN_\tau$. For every $a\in\IN\setminus\{1\}$ there exists a multiplicative bijection $\mu_a$ of $\IN$ such that $h(a^n)=h(a)^{\mu_a(n)}$ for any $n\in\w$.
\end{lemma}

\begin{proof} Given any $a\in\IN$, let $b=h(a)$ and consider the families $X:=\{a^{n\IN}:n\in\IN\}$ and $Y:=\{b^{n\IN}:n\in\IN\}$ of monogenic subsemigroups of $\IN$. The families $X,Y$ are endowed with the partial order of inclusion. This partial order turns $X,Y$ into lattices. 

In the lattices $X,Y$ consider the subsets $X_\Pi:=\{a^{p\IN}:p\in\Pi\}$ and $Y_\Pi:=\{b^{p\IN}:p\in\Pi\}$ and observe that $X_\Pi$ and $Y_\Pi$ are the sets of maximal elements of the sets $X\setminus\{a^\IN\}$ and $Y\setminus\{b^\IN\}$ in the lattices $X,Y$, respectively. It can be shown that each isomorphism $f:X\to Y$ between the lattices $X,Y$ is induced by some multiplicative bijection $\mu:\IN\to\IN$ in the sense that $f(a^{n\IN})=b^{\mu(n)\IN}$ for all $n\in\IN$. 

 By Lemma~\ref{l:monogen}, the homeomorphism $h$ induces a lattice isomorphism $$\tilde h:X\to Y,\;\;\tilde h:a^{n\IN}\mapsto h(a^n)^\IN$$of the lattices $X,Y$. So, there exists a
 multiplicative bijection $\mu_a:\IN\to\IN$ such that $\tilde h(a^{n\IN})=b^{\mu_a(n)\IN}$ for every $n\in\IN$. Then
 $h(a^n)^\IN=\tilde h (a^{n\IN})=b^{\mu_a(n)\IN}$ and hence $h(a^n)=b^{\mu_a(n)}$ for all $n\in\IN$.
 \end{proof}
 
Next, we shall prove that $\mu_a=\mu_b$ for any elements $a,b\in\IN$. For this we use the following lemma, proved by joint efforts of Mathoverflow users {\tt Fran\c cois Brunault} and {\tt so-called friend Don}, see \cite{MO}.

\begin{lemma}\label{l:Brunault} For any numbers $b\in\Pi$ and $a\in\IN\setminus\{x^b:x\in\IN\}$ there exist infinitely many prime numbers $p\in 1+b\IN$ such that $a^{(p-1)/b}\ne 1\mod p$.
\end{lemma}

\begin{proof} The choice of $a$ ensures that the equation $x^b=a$ has no solutions in $\IQ$. Then we can apply the Grunwald-Wang Theorem \cite[Ch.X]{AT} and conclude that the set $P$ of prime numbers $p$ for which the equation $x^b=a \mod p$ has no solutions is infinite. We claim that $P\setminus \Pi_a\subset 1+b\IZ$. In the opposite case, we could find a prime number $p\in P$ with $p\notin \Pi_a\cup(1+b\IZ)$ and conclude that $p-1$ is not divisible by the prime number $b$ and hence $b$ is coprime with $p-1$. It is well-known that the multiplicative group $\IZ_p^*$ of the finite field $\IZ_p:=\IZ/p\IZ$ is cyclic of order $p-1$. Since $b$ is coprime with $p-1$, the map $\IZ_p^*\to\IZ^*_p$, $x\mapsto x^b$, is bijective, which implies that the equation $(x^b=a \mod p)$ has a solution. But this contradicts the choice of $p\in P$. This contradiction completes the proof the inclusion $P\setminus \Pi_a\subset 1+b\IZ$.

It remains to prove that for each $p\in P\cap (1+b\IZ)$ we have $a^{(p-1)/b}\ne 1\mod p$. Since the group $\IZ_p^*$ is cyclic, there exists a positive number $g<p$ such that the coset $g+p\IZ$ is a generator of $\IZ_p^*$. Then $a=g^k\mod p$ for some $k<p$. Assuming that $a^{(p-1)/b}=1\mod p$, we would conclude that $g^{k(p-1)/b}=1 \mod p$ and hence $k(p-1)/b\in(p-1)\IZ$. Then $k/b$ is integer and hence $a=g^{k}=(g^{k/b})^b\mod p$, which contradicts the choice of $p\in P$.
\end{proof}
 
 \begin{lemma}\label{l:finall} For any homeomorphism $h$ of the Golomb space $\IN_\tau$ there exists a multiplicative bijection $\mu$ on $\IN$ such that  $h(a^n)=h(a)^{\mu(n)}$ for any $a,n\in\IN$.
 \end{lemma}
 
 \begin{proof} By Lemma~\ref{l:multi}, for every $a\in\IN\setminus\{1\}$ there exists a multiplicative bijection $\mu_a$ of $\IN$ such that $h(a^n)=h(a)^{\mu_a(n)}$ for all $n\in\IN$. Let $P:=\{n^{k+1}:n,k\in\IN\}$ be the set of all powers of natural numbers.
 
First we prove that for every $a\in \IN\setminus P$ and a prime $n$ there exists a  neighborhood $U_a\subset\IN_\tau$ of $a$ such that $\mu_a(n)=\mu_r(n)$ for every prime number $r\in U_a$. 
 
Observe that the image $q:=\mu_a(n)$ of the prime number $n$ under the  multiplicative bijection $\mu_a$ of $\IN$ is prime. By Lemma~\ref{l:Brunault}, there exists a prime number $p\in (1+q\IN)\setminus(\Pi_a\cup \Pi_{h(a^n)})$ such that $a^{(p-1)/q}\ne 1\mod p$.  By the continuity of the map $f:\IN\to\IN$, $f:x\mapsto h(x^n)$, the point $a$ has a neighborhood $V\in\tau_a$ such that $f(V)\subset f(a)+p\IN_0=h(a^n)+p\IN_0$. We claim that the neighborhood $U_a:=V\cap(a+p\IN_0)$ has the required property. To derive a contradiction, assume that $U_a$ contains a prime number $r\in U_a$ such that $\mu_r(n)\ne\mu_a(n)$. 
 It follows from $r\in a+p\IN_0$ that $r^{\mu_r(n)}-a^{\mu_r(n)}\in p\IZ$ and $$r^{\mu_r(n)}=h(r^n)=f(r)\in f(U_a)\subset f(a)+p\IN_0=h(a^n)+p\IN_0\subset a^{\mu_a(n)}+p\IZ.$$ Then $a^{\mu_a(n)}-a^{\mu_r(n)}=a^{\mu_a(n)}-r^{\mu_r(n)}+r^{\mu_r(n)}-a^{\mu_r(n)}\in p\IZ$ and hence $a^{|\mu_a(n)-\mu_r(n)|}-1$ is divisible by $p$.
 Since the multiplicative group $\IZ_p^*$ of the field $\IZ_p:=\IZ/p\IZ$ is cyclic of order $p-1$, the inclusion $a^{|\mu_a(n)-\mu_r(n)|}-1\in p\IZ$ implies that $\mu_r(n)-\mu_a(n)=\frac{p-1}{d}+(p-1)k$ for some divisor $d$ of $p-1$ and some $k\in\IZ$. The choice of $p$ in $1+\mu_a(n)\IN$ ensures that $p-1=\mu_a(n)\cdot m$ for some $m\in\IZ$.  Then $$\mu_r(n)=\mu _a(n)+\frac{p-1}d+(p-1)k= \mu_a(n)+\tfrac{m\mu_a(n)}{d}+\mu_a(n)km.$$
 Let us show that the prime number $q=\mu_a(n)$ does not divide $d$.
 Otherwise $\frac{p-1}d$ divides $\frac{p-1}{q}$ and $a^{(p-1)/d}=1\mod p$ would imply $a^{(p-1)/q}=1\mod p$, which contradicts the choice of $p$. This contradiction shows that the number $q=\mu_a(n)$ does not divide $d$ and hence $\frac{m}{d}$ is integer. Then $\mu_r(n)=\mu_a(n)(1+\frac{m}d+km)$ and we obtain that $\mu_r(n)$ is not prime, which is a desired contradiction showing that $\mu_a(n)=\mu_r(n)$ for every $r\in \Pi\cap U_a$.
 \smallskip
 
 Now we can prove that $\mu_a(n)=\mu_{b}(n)$ for arbitary $a,b\in \IN\setminus P$ and $n\in\Pi$. Since the Golomb space $\IN_\tau$ is connected, there exists a chain of points $a=a_0,a_1,\dots,a_m=b$ such that for every $i<m$ the intersection $U_{a_{i}}\cap U_{a_{i+1}}$ is not empty and hence contains some prime number $r_i$. Then $\mu_{a_i}(n)=\mu_{r_i}(n)=\mu_{a_{i+1}}(n)$ for all $i<m$ and hence $\mu_a(n)=\mu_{a_0}(n)=\mu_{a_n}(n)=\mu_{b}(n)$. 
 
Since the functions $\mu_a$ and $\mu_b$ are multiplicative, the equality $\mu_a(n)=\mu_b(n)$ holding for all prime $n\in\Pi$ implies the equality $\mu_a=\mu_b$. Choose any number $c\in\IN\setminus P$ and put $\mu:=\mu_c$. 

Finally, we prove that $\mu_a=\mu$ for any $a\in\IN$. This equality has been proved for any $a\in\IN\setminus P$. So, assume that $a\in P$. In this case $a=\alpha^k$ for some $\alpha\in \IN\setminus P$ and some $k\in\IN$.
We claim that $h(\alpha)\notin P$. Otherwise we could find numbers $\beta$ and $l>1$ such that $h(\alpha)=\beta^l$. By Lemma~\ref{l:monogen}, $\alpha\in h^{-1}(\beta^{\IN})\setminus\{h^{-1}(\beta)\}$, which means that $\alpha\in P$. But this contradicts the  choice of $\alpha$. This contradiction shows that $h(\alpha)\notin P$.

Then for any $n\in\IN$ we get
\begin{multline*}
h(a)^{\mu_a(n)}=h(a^n)=h(\alpha^{kn})=h(\alpha)^{\mu_{h(\alpha)}(kn)}=h(\alpha)^{\mu(kn)}=\\
=h(\alpha)^{\mu(k)\mu(n)}=(h(\alpha)^{\mu(k)})^{\mu(n)}=h(\alpha^k)^{\mu(n)}=h(a)^{\mu(n)}.
\end{multline*}
  \end{proof}

\begin{problem} Let $h$ be a homeomorphism of the Golomb space $\IN_\tau$. Is it true that $h(x^n)=h(x)^n$ for any $x\in\IN$?
\end{problem}

Theorem~\ref{mainH} implies that the Golomb space $\IN_\tau$ is not topologically homogeneous. On the other hand, we do not know the answer to the following intriguing problem (posed also in \cite{Ban}).

\begin{problem} Is the Golomb space rigid?
\end{problem}

We recall that a topological space $X$ is {\em rigid} if each homeomorphism $h:X\to X$  is equal to the identity map of $X$.

\begin{remark} A counterpart of the Golomb topology on domains (= commutative rings without zero divisors) was introduced and studied by Knopfmacher and Porubsk\'y \cite{KP}. In their recent preprint \cite{CLP} Clark, Lebowitz-Lockard and Pollack extended some results of this paper to the Golomb topology on domains.
\end{remark}

\section{Acknowledgement}

The authors express their sincere thanks to Patrick Rabau for a valuable remark on the formula for the closure in Lemma~\ref{basic}, to Gergely Harcos for the reference to the very helpful paper \cite{SL} on the Chebotar\"ev and Frobenius density theorems, and to the mathoverflow users Fran\c cois Brunault, David Loeffler and Paul Pollack for their help with the proof of Lemma~\ref{l:Brunault}.

\end{document}